\documentclass{birkjour}

 \newtheorem{Th}{Theorem}[section]
 \newtheorem{Cor}[Th]{Corollary}
 \newtheorem{Lem}[Th]{Lemma}
 
 \theoremstyle{definition}
 \newtheorem{Def}[Th]{Definition}
 \theoremstyle{remark}
 \newtheorem{Rem}[Th]{Remark}
 
 \numberwithin{equation}{section}
 
 \newcommand{\R}{\mathbb{R}}

\newcommand{\Z}{\mathbb Z}

\newcommand{\cC}{{\mathcal C}}
\newcommand{\cF}{{\mathcal F}}
\newcommand{\cJ}{{\mathcal J}}
\newcommand{\cN}{{\mathcal N}}
\newcommand{\weakto}{\rightharpoonup}
\newcommand{\Ga}{\Gamma}

\DeclareMathOperator*{\supp}{supp}

\begin{document}

%
%
%
%
%
%
%
%
%

\title[Non-local to local transition]
 {Non-local to local transition for ground states of fractional
  Schr\"{o}dinger equations on $\R^N$}

\author[Bieganowski]{Bartosz Bieganowski}

\address{%
Nicolaus Copernicus
  University \\ Faculty of Mathematics and Computer Science \\ ul. Chopina
  12/18, 87-100 Toru\'n, Poland}

\email{bartoszb@mat.umk.pl}

\author[Secchi]{Simone Secchi (corresponding author)}
\address{Dipartimento di
  Matematica e Applicazioni \\ Universit\`a degli Studi di
  Milano-Bicocca \\
  via Roberto Cozzi 55, I-20125, Milano, Italy}
\email{simone.secchi@unimib.it}


\date{\today}

\begin{abstract}
  We consider the nonlinear fractional problem
  \begin{align*}
    (-\Delta)^{s} u + V(x) u = f(x,u) &\quad \hbox{in $\R^N$}
\end{align*}
We show that ground state solutions converge (along a subsequence) in
$L^2_{\mathrm{loc}} (\R^N)$, under suitable conditions on $f$ and $V$,
to a ground state solution of the local problem as $s \to 1^-$.
\end{abstract}

\maketitle

\section{Introduction}

The aim of this paper is to analyse the asymptotic behavior of
least-energy solutions to the fractional Schr\"odinger problem
\begin{equation} \label{eq:1.1}
  \begin{cases}
    (-\Delta)^{s} u + V(x) u = f(x,u) &\quad \hbox{in $\mathbb{R}^N$} \\
    u \in H^s (\R^N),
\end{cases}
\end{equation}
under suitable assumptions on the scalar potential
$V \colon \mathbb{R}^N \to \mathbb{R}$ and on the nonlinearity
$f \colon \mathbb{R}^N \times \mathbb{R} \to \mathbb{R}$.  We recall
that the fractional laplacian is defined as the principal value of a
singular integral via the formula
\begin{align*}
  (-\Delta)^s u (x) = C(N,s) \lim_{\varepsilon \to 0} \int_{\mathbb{R}^N \setminus B_\varepsilon(x)} \frac{u(x)-u(y)}{|x-y|^{N+2s}}\, dy
\end{align*}
with
\begin{align*}
  \frac{1}{C(N,s)} = \int_{\mathbb{R}^N} \frac{1-\cos \zeta_1}{|\zeta|^{N+2s}}\, d\zeta_1 \cdots d\zeta_N.
\end{align*}
This formal definition needs of course a function space in which
problem \eqref{eq:1.1} becomes meaningful: we will come to this issue
in Section \ref{sect:variational-setting}.

\medskip

Several models have appeared in recent years that involve the use of
the fractional laplacian. We only mention elasticity, turbulence,
porous media flow, image processing, wave propagation in heterogeneous
high contrast media, and stochastic models: see \cite{B,DPV,V,G}.

\medskip

Instead of \emph{fixing} the value of the parameter $s \in (0,1)$, we
will start from the well-known identity (see \cite[Proposition
4.4]{DNPV})
\begin{align}
\lim_{s \to 1^-} (-\Delta)^s u = -\Delta u \label{eq:1.1.0}
\end{align}
valid for any $u \in C_0^\infty(\mathbb{R}^N)$, and investigate the
convergence properties of solutions to \eqref{eq:1.1} as $s \to 1^-$.

In view of \eqref{eq:1.1.0}, it is somehow natural to conjecture that
solutions to \eqref{eq:1.1} converge to solutions of the problem
\begin{align} \label{eq:1.2}
\begin{cases}
-\Delta u + V(x) u = f(x,u) &\hbox{in $\R^N$,} \\
u  \in H^1 (\R^N).
\end{cases}
\end{align}
We do not know if this conjecture is indeed correct with this degree
of generality. 

In this paper we will always assume that both $V$ and $f$ are $\mathbb{Z}^N$-periodic in the space variables. Hence equations \eqref{eq:1.1} and \eqref{eq:1.2} are
invariant under $\mathbb{Z}^N$-translations, and their solutions are
not unique.
We will prove that --- up to $\Z^N$-translations and along
a subsequence --- \emph{least-energy} solutions of \eqref{eq:1.1}
converge to a ground state solution to the local problem
\eqref{eq:1.2}. Our result is a continuation of the previours paper
\cite{BS}, in which we consider the equation on a bounded domain and
extend the very recent analysis of Biccari \emph{et al.}
(see~\cite{BHS}) in the \emph{linear} case for the Poisson problem to
the semilinear case. See also \cite{BC}.

\medskip

We collect our assumptions.
\begin{itemize}
\item[(N)] $N \geq 3$, $1/2 < s < 1$;
\item[(V)] $V \in L^\infty(\R^N)$ is $\Z^N$-periodic and $\inf_{\R^N} V >  0$;
\item[(F1)] $f \colon \R^N \times \mathbb{R} \to \mathbb{R}$ is a
  Carath\'eodory function, namely $f(\cdot, u)$ is measurable for any
  $u \in \R$ and $f(x, \cdot)$ is continuous for a.e. $x \in
  \R^N$. Moreover $f$ is $\Z^N$-periodic in $x \in \R^N$ and there are
  numbers $C > 0$ and $p \in \left(2, \frac{2N}{N-1} \right)$ such
  that
\begin{align*}
|f(x,u)| \leq C (1 + |u|^{p-1})
\end{align*}
for $u \in \R$ and a.e. $x \in \R^N$.
\item[(F2)] $f(x,u)=o(u)$ as $u \to 0$, uniformly with respect to $x
  \in \R^N$.
\item[(F3)] $\lim_{|u| \to +\infty} \frac{F(x,u)}{u^2}=+\infty$
  uniformly with respect to $x \in \R^N$, where
  $F(x,u)=\int_0^u f(x,s)\, ds$.
\item[(F4)] The function $\R \setminus \{0\} \ni u \mapsto f(x,u)/u$
  is strictly increasing on $(-\infty, 0)$ and on $(0, \infty)$, for
  a.e. $x \in \R^N$.
\end{itemize}

\begin{Rem} \label{rem:1.1}
  It follows from (F1) and (F2) that for
  every~$\varepsilon > 0$ there is~$C_\varepsilon > 0$ such that
\[
|f(x,u)| \leq \varepsilon |u| + C_\varepsilon |u|^{p-1}
\]
for every~$u \in \mathbb{R}$ and a.e~$x \in \R^N$.  Furthermore,
assumption~(F4) implies the validity of the inequality
\begin{align*}
0 \leq 2 F(u) \leq f(x,u)u
\end{align*}
for every~$u \in \R$ and a.e.~$x \in \R^N$.
\end{Rem}
We can now state our main result. 
\begin{Th}\label{th:main}
  Suppose that assumptions~(N), (V), (F1)--(F4) hold. Let
  $u_s \in H^s(\R^N)$ be a ground state solution of problem
  \eqref{eq:1.1}. Then, there exists a sequence
  $\{s_n\}_n \subset (1/2, 1)$ such that $s_n \to 1$ as
  $n \to + \infty$ and there exists a sequence of translations
  $\{z_n\}_n$ such that $u_{s_n}(\cdot - z_n)$ converges in
  $L^2_{\mathrm{loc}}(\R^N)$ to a ground state solution
  $u_0 \in H^1(\R^N)$ of the problem \eqref{eq:1.2}.
\end{Th}

\section{The variational setting}\label{sect:variational-setting}

In this section we collect the basic tools from the theory of
fractional Sobolev spaces we will need to prove our results. For a
thorough discussion, we refer to \cite{MBRS,DNPV} and to the
references therein.

\medskip

For $0<s<1$, we define a Sobolev space on $\R^N$ as
\begin{align*}
  H^s(\R^N)= \left\{ u \in L^2(\R^N) \mid \int_{\R^N \times \R^N} \frac{|u(x)-u(y)|^2}{|x-y|^{N+2s}} \, dx\, dy < + \infty \right\},
\end{align*}
endowed with the norm
\begin{align*}
  \|u\|_{H^s(\R^N)}^2 = \|u\|_{L^2(\R^N)}^2 + \int_{\R^N \times \R^N} \frac{|u(x)-u(y)|^2}{|x-y|^{N+2s}} \, dx\, dy
\end{align*}
One can show that $C_0^\infty(\R^N)$ is dense in $H^s(\R^N)$. For
$u \in H^s(\R^N)$, an equivalent norm of $u$ is (see \cite[Proposition
1.18]{MBRS})
\begin{align*}
	u \mapsto \left( \|u\|_{L^2(\R^N)}^2 + \left\| (-\Delta)^{\frac{s}{2}} u \right\|_{L^2(\mathbb{R}^N)}^2 \right)^{1/2}.
\end{align*}
More explicitly, for every $u \in H^s(\R^N)$,
\begin{align*}
\int_{\mathbb{R}^N \times \mathbb{R}^N} \frac{|u(x)-u(y)|^2}{|x-y|^{N+2s}}\, dx\, dy = \frac{2}{C(N,s)} \left\| (-\Delta)^{s/2} u \right\|^2_{L^2(\mathbb{R}^N)},
\end{align*}
where
\begin{align*}
C(N,s) &=\frac{s(1-s)}{A(N,s)B(s)}, \\
A(N,s) &= \int_{\mathbb{R}^{N-1}} \frac{d \eta}{(1+|\eta|^2)^{(N+2s)/2}}, \\
B(s) &= s (1-s) \int_{\mathbb{R}} \frac{1-\cos t}{|t|^{1+2s}} \, dt.
\end{align*}
\begin{Lem}
  For every $u \in H^1(\mathbb{R}^N)$, there results
  \begin{align*}
    \lim_{s \to 1^-} \left\| (-\Delta)^{s/2} u
    \right\|_{L^2(\mathbb{R}^N)}^2 = \left\| \nabla u
    \right\|_{L^2(\mathbb{R}^N)}^2.
  \end{align*}
\end{Lem}
\begin{proof}
  From \cite[Proposition 3.6]{DNPV}, we know that
  \begin{align*}
    \left\| (-\Delta)^{s/2} u \right\|_{L^2(\mathbb{R}^N)}^2 =
    \frac{C(N,s)}{2} \int_{\mathbb{R}^N \times \mathbb{R}^N}
    \frac{|u(x)-u(y)|^2}{|x-y|^{N+2s}}\, dx\, dy.
  \end{align*}
  From \cite[Remark 4.3]{DNPV}, we know that
  \begin{align*}
    \lim_{s \to 1^-} (1-s) \int_{\mathbb{R}^N \times \mathbb{R}^N}
    \frac{|u(x)-u(y)|^2}{|x-y|^{N+2s}}\, dx\, dy =
    \frac{\omega_{N-1}}{2N} \left\| \nabla u \right\|_{L^2(\mathbb{R}^N)}^2.
  \end{align*}
  Therefore, recalling \cite[Corollary 4.2]{DNPV},
  \begin{align*}
    \lim_{s \to 1^-} \left\| (-\Delta)^{s/2} u
    \right\|_{L^2(\mathbb{R}^N)}^2 &= \lim_{s \to 1^-}
                                     \frac{C(N,s)}{2(1-s)}
                                     \left( (1-s) \int_{\mathbb{R}^N \times \mathbb{R}^N}
    \frac{|u(x)-u(y)|^2}{|x-y|^{N+2s}}\, dx\, dy \right) \\
    &= \frac{1}{2} \frac{4N}{\omega_{N-1}}
      \frac{\omega_{N-1}}{2N} \left\| \nabla u
      \right\|_{L^2(\mathbb{R}^N)}^2 = \left\| \nabla u \right\|_{L^2(\mathbb{R}^N)}^2.
    \end{align*}
  \end{proof}
  On $H^s (\R^N)$ we introduce a new norm
\begin{align}\label{norm-s}
\|u\|_s^2 := \left\| (-\Delta)^{s/2} u \right\|_{L^2(\mathbb{R}^N)}^2 + \int_{\R^N} V(x) u^2 \, dx, \quad u \in H^s (\R^N),
\end{align}
which is, under (V), equivalent to $\| \cdot \|_{H^s
  (\R^N)}$. Similarly we introduce the norm on $H^1 (\R^N)$ by putting
\begin{align}\label{norm-1}
\|u\|^2 := \int_{\R^N} |\nabla u|^2 + V(x) u^2 \, dx, \quad u \in H^1(\R^N).
\end{align}

\begin{Cor} \label{cor:2.3}
For every $u \in H^1(\R^N)$ we have
  \begin{align*}
    \lim_{s \to 1^-} \|u\|_s = \|u\|.
  \end{align*}
\end{Cor}
The following convergence result will be used in the sequel.
\begin{Lem}
  For every $ \varphi \in C_0^\infty(\R^N)$, there results
  \begin{align*}
    \lim_{s \to 1^-} \left\| (-\Delta)^{s} \varphi - (-\Delta) \varphi
    \right\|_{L^2(\R^N)} =0.
  \end{align*}
\end{Lem}
\begin{proof}
  We notice that
  \begin{multline*}
    \left\| (-\Delta)^{s} \varphi - (-\Delta) \varphi
    \right\|_{L^2(\mathbb{R}^N)} = \left\| \mathcal{F}^{-1}_\xi \left( \left(
    |\xi|^{2s} - |\xi|^2 \right) \hat{\varphi}(\xi) \right)
                             \right\|_{L^2(\mathbb{R}^N)}
    \\ \leq C \left\| \left(
    |\cdot|^{2s} - |\cdot|^2 \right) \hat{\varphi}
                             \right\|_{L^2(\mathbb{R}^N)}
  \end{multline*}
  where $C>0$ is a constant, independent of $s$, that depends on the
  definition of the Fourier transform $\cF$. It is now easy to conclude,
  since the Fourier transform of a test function is a rapidly
  decreasing function.
\end{proof}

We will need some precise information on the embedding constant for
fractional Sobolev spaces.
\begin{Th}[\cite{KT}]\label{ThKT}
  Let $N>2s$ and $2_s^*=2N/(N-2s)$. Then
  \begin{align*} 
    \|u\|_{L^{2_s^*} (\R^N)}^2 \leq \frac{\Gamma \left( \frac{N-2s}{2}
    \right)}{\Gamma \left( \frac{N+2s}{2} \right)} \left| \mathbb{S}
    \right|^{-\frac{2s}{N}} \| (-\Delta)^{s/2} u \|^2_{L^2 (\R^N)}
  \end{align*}
  for every $u \in H^s(\mathbb{R}^N)$, where $\mathbb{S}$ denotes the
  $N$-dimensional unit sphere and $|\mathbb{S}|$ its surface area.
\end{Th}

The following inequality in an easy consequence of Theorem \ref{ThKT}, see also \cite[Lemma 2.7]{BS}.

\begin{Lem}\label{lemma-constant-independednt}
  Let $N \geq 3$ and $q \in [2,2N/(N-1)]$. Then there exists a
  constant $C=C(N,q)>0$ such that, for every $s \in [1/2,1]$ and every
  $u \in H^s(\R^N)$, we have
	\begin{align*}
	\|u\|_{L^{q} (\R^N)} \leq C(N,q) \| u \|_{s}.
	\end{align*}
\end{Lem}

\begin{Def}
  A weak solution to problem \eqref{eq:1.1} is a function
  $u \in H^s(\R^N)$ such that
	\begin{align*}
	\langle (-\Delta)^{s/2} u \mid (-\Delta)^{s/2} \varphi \rangle_{L^2(\mathbb{R}^N)} + \int_{\R^N} V(x) u \varphi \, dx = \int_{\R^N} f(x,u) \varphi\, dx
	\end{align*}
	for every $\varphi \in H^s(\R^N)$.
\end{Def}
Weak solutions are therefore critical points of the associated energy
functional $\cJ_s \colon H^s (\R^N) \rightarrow \R$ defined by
\begin{align*}
  \cJ_s (u) = \frac12 \left\| (-\Delta)^{s/2} u \right\|_{L^2(\mathbb{R}^N)}^2 + \frac12
  \int_{\R^N} V(x) u^2 \, dx - \int_{\R^N} F(x,u) \, dx.
\end{align*}

We recall also the definition of a weak solution in the local case.

\begin{Def} 
  A weak solution to problem \eqref{eq:1.2} is a function
  $u \in H^1 (\R^N)$ such that
\begin{align*}
\int_{\R^N} \nabla u \cdot \nabla \varphi \, dx + \int_{\R^N} V(x) u \varphi \, dx = \int_{\R^N} f(x,u) \varphi \, dx
\end{align*}
for every $\varphi \in H^1(\R^N)$.
\end{Def}

For the local problem \eqref{eq:1.2} we put
$\cJ\colon H^1 (\R^N) \rightarrow \R$
\begin{align} \label{eq:2.2}
  \cJ(u) = \frac12 \int_{\R^N} |\nabla u|^2 + V(x) u^2 \, dx - \int_{\R^N} F(x,u) \, dx.
\end{align}

Recalling the notation \eqref{norm-s} and \eqref{norm-1}, we can
rewrite our functionals in the form
\begin{align*}
\cJ_s (u) &= \frac12 \|u\|_s^2 - \int_{\R^N} F(x,u) \, dx, \quad u \in H^s (\R^N), \\
\cJ (u) &= \frac12 \|u\|^2 - \int_{\R^N} F(x,u) \, dx, \quad u \in H^1 (\R^N).
\end{align*}

\section{Uniform Lions' concentration-compactness principle}

Since the summability exponent of our space is not fixed, we need a
``uniform'' version of a celebrated result by P.-L. Lions.
\begin{Th}\label{unif-lions}
  Let $r > 0$, $2 \leq q < \frac{2N}{N-1}$ and $N \geq 3$. Suppose
  moreover that $\{ s_n \}_n \subset (1/2, 1)$,
  $u_n \in H^{s_n} (\R^N)$ and
\begin{align*}
\| u_n \|_{s_n} \leq M,
\end{align*}
where $M > 0$ does not depend on $s_n$. If
\begin{align*}
\lim_{n \to +\infty} \sup_{y \in \R^N} \int_{B(y,r)} |u_n|^q \, dx = 0
\end{align*}
then $u_n \to 0$
in $L^p (\R^N)$ for all $p \in \left( 2, \frac{2N}{N-1} \right)$.
\end{Th}

\begin{proof}
Let $t \in \left( q, \frac{2N}{N-1} \right)$. Then
\begin{align*}
\| u_{n} \|_{L^t (B(y,r))} &\leq \| u_{n} \|_{L^q (B(y,r))}^{1 - \lambda} \| u_{n} \|_{L^{\frac{2N}{N-1}} (B(y,r))}^\lambda \\
&\leq C \| u_{n} \|_{L^q (B(y,r))}^{1 - \lambda} \| u_{n} \|_{s_n}^\lambda,
\end{align*}
where $C > 0$ is independent of $s_n$ and
$\lambda = \frac{t-q}{\frac{2N}{N-1}-q} \frac{2N}{(N-1) t}$. Choose
$t$ such that $\lambda = \frac{2}{t}$. Then
\begin{align*}
\int_{ B(y,r)} |u_{n}|^t \, dx \leq C^t \| u_{n} \|_{L^q (B(y,r))}^{(1 - \lambda)t} \| u_{n} \|_{s_n}^{2}.
\end{align*}
Covering space $\R^N$ by balls of radius $r$, in a way that each point
is contained in at most $N+1$ balls, we get
\begin{align*}
\int_{\R^N} |u_{n}|^t \, dx &\leq (N+1) C^t \sup_{y \in \R^N} \left( \int_{B(y,r)} |u_{n}|^q \, dx \right)^{\frac{(1 - \lambda)t}{q}} \| u_{n} \|_{s_n}^{2} \\
&\leq (N+1) M^2 C^t \sup_{y \in \R^N} \left( \int_{B(y,r)} |u_{n}|^q \, dx \right)^{\frac{(1 - \lambda)t}{q}} \to 0.
\end{align*}
Hence $u_{n} \to 0$ in $L^t (\R^N)$. Note that
\begin{align*}
\| u_n\|_{L^2 (\R^N)}^2 \leq D \|u_n\|_{s_n}^2 \leq D M^2,
\end{align*}
where $D$ does not depend on $s_n$ and $n$. Similarly, from Lemma \ref{lemma-constant-independednt}, there follows that $\{ u_n\}_n$ is bounded in $L^{\frac{2N}{N-1}} (\R^N)$. From the interpolation inequality, since $\{ u_n \}_n$ is bounded in $L^2 (\R^N)$ and in $L^{\frac{2N}{N-1}} (\R^N)$, we obtain $u_n \to 0$ in $L^p (\R^N)$ for all $p \in \left( 2, \frac{2N}{N-1} \right)$.
\end{proof}


Finally, we extend the locally compact embedding into Lebesgue spaces
in a uniform way.

\begin{Th}\label{local-compactness}
  Let $\{s_n\}_n$ be a sequence such that $1/2 < s_n < 1$ and
  $s_n \rightarrow1$, and let
  $\{v_{s_n}\}_n \subset H^{s_n}(\mathbb{R}^N)$ be such that
\begin{align*} 
M = \sup_n \|v_{s_n}\|_{s_n} < \infty.
\end{align*}
Then the sequence $\{v_{s_n}\}_n$ converges, up to a subsequence, to
some $v\in H^1 (\mathbb{R}^N)$ in
$L_{\mathrm{loc}}^q(\mathbb{R}^N)$ for every $q \in [2,2N/(N-1))$, and
pointwise almost everywhere.
\end{Th}

\begin{proof}
  Note that $H^{s_n} (\R^N) \subset H^{1/2} (\R^N)$ and
\begin{align*}
\| \cdot \|_{1/2} \leq C \| \cdot \|_{s_n}
\end{align*}
where $C>0$ does not depend on $s_n$ (and therefore also on $n$): see
for instance \cite[Proposition 1.1]{MBRS}.  In particular, for every
$n\in \mathbb{N}$ we have
\begin{align}\label{vs-bounded}
\|v_{s_n} \|_{1 / 2} \leq C \|v_{s_n} \|_{s_n} \leq CM.
\end{align}
Thus $\{ v_{s_n} \}_n$ is bounded in $H^{1/2} (\R^N)$. Hence, passing
to a subsequence, there exists a function $v$ such that
$v_{s_n} \weakto v$ in $H^{1/2} (\R^N)$, $v_{s_n} \to v$ pointwise
almost everywhere, and $v_{s_n} \to v$ in $L^q_{\mathrm{loc}} (\R^N)$
for every $q \in [2, 2N/(N-1) )$.
From \cite[Corollary 7]{BBM} it follows that
$v \in H^1_{\mathrm{loc}} (\R^N)$.
To complete the proof, we need to show that $v \in H^1 (\R^N)$.

Let $\widehat{v_{s_n}}$ denote the Fourier transform of $v_{s_n}$, similarly for $\widehat{v}$. We may assume, without loss of generality, that $\widehat{v_{s_n}} \weakto \widehat{v}$ in $L^2 (\R^N)$. Note that \eqref{vs-bounded} implies that
\begin{gather*}
\sup_{n} \int_{\R^N} (1+|\xi|^2)^{s_n} \left| \widehat{v_{s_n}} \right|^2 \, d\xi \leq K
\end{gather*}
for some constant $K > 0$. For $1/2<t  \leq 1$ we define
\begin{gather*}
B_t := \left\{ w \in L^2 (\R^N) \mid \ \int_{\R^N} (1+|\xi|^2)^t |w(\xi)|^2 \, d\xi \leq K \right\}.
\end{gather*}
First of all, we observe that
\begin{gather} \label{eq:3.2}
\bigcap_{1/2 < t < 1}  B_t = B_1.
\end{gather}
Indeed, for any $1/2 <t < 1$ we have $(1 + |\xi|^2)^{t} \leq 1 + |\xi|^2$. Take $w \in B_1$ and note that
\begin{gather*}
\int_{\R^N} (1 + |\xi|^2)^{t} \, |w(\xi)|^2 d\xi \leq \int_{\R^N} (1 + |\xi|^2) \, |w(\xi)|^2 d\xi \leq K.
\end{gather*}
Hence $w \in B_t$ for any $t < 1$. Thus
\begin{gather*}
\bigcap_{1/2 < t<1} B_t
\supset
B_1.
\end{gather*}
On the other hand, fix $w \in \bigcap_{1/2 < t<1} B_t$. Take any sequence $t_n \to 1^-$ with $t_n > 1/2$. Then obviously
\begin{gather*}
\liminf_{n\to+\infty} (1 + |\xi|^2)^{t_n} |w(\xi)|^2 = (1 + |\xi|^2) |w(\xi)|^2
\end{gather*}
and  Fatou's lemma yields
\begin{gather*}
\int_{\R^N} (1 + |\xi|^2) \, |w(\xi)|^2 d\xi \leq \liminf_{n\to+\infty} \int_{\R^N} (1 + |\xi|^2)^{t_n} \, |w(\xi)|^2 d\xi \leq K.
\end{gather*}
Hence $w \in B_1$, or 
\begin{gather*}
\bigcap_{1/2 < t<1} B_t
\subset
B_1,
\end{gather*}
and \eqref{eq:3.2} is proved.
Fix now any $t \in \left( 1/2, 1 \right)$ and choose $n_0$ such that $s_n > t$ for all $n \geq n_0$. Then
\begin{gather*}
(1+|\xi|^2)^{t} \leq (1+|\xi|^2)^{s_n} \quad \hbox{for every $\xi \in \R^N$,}
\end{gather*}
and
\begin{gather*}
\int_{\R^N} (1+|\xi|^2)^{t} \left| \widehat{v_{s_n}} \right|^2 \, d\xi \leq \int_{\R^N} (1+|\xi|^2)^{s_n} \left| \widehat{v_{s_n}} \right|^2 \, d\xi \leq K \quad \mbox{for } n \geq n_0.
\end{gather*}
Hence $\widehat{v_{s_n}} \in B_t$ for $n \geq n_0$. Each~$B_t$ is a closed and convex subset in $L^2(\R^N)$, and from \cite[Theorem 3.7]{Brezis} it is also weakly closed. Hence $\widehat{v} \in B_t$. Therefore, recalling \eqref{eq:3.2},
\begin{gather*}
\widehat{v} \in \bigcap_{\frac12 < t < 1} B_t =  \left\{ w \in L^2 (\R^N) \mid \ \int_{\R^N} (1+|\xi|^2) |w(\xi)|^2 \, d\xi \leq K \right\}.
\end{gather*}
This implies that 
\begin{gather*}
\int_{\R^N} (1+|\xi|^2) |\hat{v}(\xi)|^2 \, d\xi \leq K < +\infty
\end{gather*}
and $v \in H^1 (\R^N)$.

\end{proof}

\section{Existence of ground states}

It is easy to check that the energy functional $\cJ$ has the
mountain-pass geometry. In particular, there is radius $r > 0$ such
that
\begin{align*}
\inf_{\|u\| = r} \cJ(u) > 0.
\end{align*}
The following existence result is well-known in the literature, and
has been shown in various ways, see e.g. \cite{BM, SzW, SzW-Handbook,
  dPKSz}.

\begin{Th}
  Suppose that assumptions~(N), (V), (F1)--(F4) hold. Then there
  exists a ground state solution $u_0 \in H^1 (\R^N)$ to
  \eqref{eq:1.2}, i.e. a critical point of the functional $\cJ$ given
  by \eqref{eq:2.2} such that
\begin{align*}
\cJ(u_0) = \inf_{\cN} \cJ = \inf_{u \in H^1 (\R^N) \setminus \{0\}} \sup_{t \geq 0} \cJ(tu) = \inf_{\gamma \in \Gamma} \sup_{t \in [0,1]} \cJ(\gamma(t)),
\end{align*}
where $\cN$ is the so-called Nehari manifold
\begin{align*}
\cN := \{ u \in H^1 (\R^N) \setminus \{0\} \mid \cJ'(u)(u) = 0\}
\end{align*}
and
\begin{align*}
\Ga := \{ \gamma \in \cC([0,1], H^1 (\R^N)) \mid \gamma(0)=0, \ \| \gamma(1) \| > r, \ \cJ(\gamma(1)) < 0 \}.
\end{align*}
\end{Th}

The same methods can be applied also in the nonlocal case, and the
following existence result can be shown, see e.g. \cite{Bi, S1,
  S2}. In what follows, $r_s > 0$ is the radius chosen so that
\begin{align*}
\inf_{\|u\|_s = r_s} \cJ_s(u) > 0.
\end{align*}

\begin{Th}
  Suppose that assumptions~(N), (V), (F1)--(F4) hold and
  $1/2 < s < 1$. Then there exists a ground state solution
  $u_s \in H^s (\R^N)$ to \eqref{eq:1.1}, i.e. a critical point of the
  functional $\cJ_s$ given by \eqref{eq:2.2} such that
\begin{equation}\label{eq:5.1}
\cJ_s(u_s) = \inf_{\cN_s} \cJ_s = \inf_{u \in H^s (\R^N) \setminus \{0\}} \sup_{t \geq 0} \cJ_s(tu) = \inf_{\gamma \in \Gamma_s} \sup_{t \in [0,1]} \cJ_s(\gamma(t)),
\end{equation}
where $\cN_s$ is the corresponding Nehari manifold
\begin{align*}
\cN_s := \{ u \in H^s (\R^N) \setminus \{0\} \mid \cJ_s'(u)(u) = 0\}
\end{align*}
and
\begin{align*}
\Ga_s := \{ \gamma \in \cC([0,1], H^s (\R^N)) \mid \gamma(0)=0, \ \| \gamma(1) \| > r_s, \ \cJ_s(\gamma(1)) < 0 \}.
\end{align*}
\end{Th}

\section{Non-local to local transition}
For any $s \in (1/2,1)$ we define
\begin{align*}
c_s := \inf_{\cN_s} \cJ_s > 0.
\end{align*}
Similarly, we put also
\begin{align*}
c := \inf_{\cN} \cJ > 0.
\end{align*}
For any $v \in H^s (\R^N) \setminus \{0\}$ let $t_s(v) > 0$ be the
unique positive real number such that $t_s (v) \in \cN_s$. Then we put
$m_s(v) := t_s (v) v$ .
\begin{Lem}\label{lem:limsup}
	There results
\begin{align*}
\limsup_{s \to 1^-} c_s \leq c.
\end{align*}
\end{Lem}
\begin{proof}
  Take $u \in H^1 (\R^N) \subset H^s (\R^N)$ as a ground state
  solution of \eqref{eq:1.2}, in particular $u \in \cN$ and
  $\cJ(u) = c$, where $\cJ$ is given by \eqref{eq:2.2}. Consider the function $m_s (u) \in \cN_s$. Obviously
\begin{align*}
c_s \leq \cJ_s (m_s (u)).
\end{align*}
Hence
\begin{multline*}
  \limsup_{s \to 1^-} c_s \leq \limsup_{s \to 1^-} \cJ_s (m_s(u)) \\ =
  \limsup_{s \to 1^-}  \left\{ \cJ_s (m_s(u))  - \frac12 \cJ_s '
    (m_s(u))  (m_s(u))  \right\} \\
  = \limsup_{s \to 1^-} \left\{ \frac12 \int_{\R^N} f(x,m_s(u))m_s(u) - 2
  F(x,m_s(u)) \, dx \right\}.
\end{multline*}
Recall that $m_s(u) = t_s u$ for some real numbers $t_s > 0$. Suppose
by contradiction that $t_s \to +\infty$ as $s \to 1^-$. Then, in view
of the Nehari identity
\begin{align*}
\| u \|_{s}^2 = \int_{\R^N} \frac{f(x,t_s u)}{t_s^2} t_s u \, dx
\geq 2 \int_{\R^N} \frac{F(x,t_s u)}{t_s^2 u^2} u^2 \, dx \to +\infty,
\end{align*}
but the left-hand side stays bounded (see Corollary \ref{cor:2.3}). Hence $(t_s)_s$ is bounded. Take any
convergent subsequence $(t_{s_n})$ of $(t_s)$, i.e. $t_{s_n} \to t_0$ as $n\to+\infty$. Obviously $t_0 \geq 0$. We will show that $t_0 \neq 0$. Indeed, suppose that $t_0 = 0$, i.e. $t_{s_n} \to 0$. Then, in view of the Nehari identity
\begin{align*}
\| u \|_{s_n}^2 = \int_{\R^N} \frac{f(x,t_{s_n} u)}{t_{s_n} u} u^2 \, dx.
\end{align*}
By Corollary \ref{cor:2.3}, $\| u \|_{s_n}^2 \to \|u\|^2 > 0$. Hence, in view of (F2),
\begin{align*}
\|u\|^2 + o(1) = \int_{\R^N} \frac{f(x,t_{s_n} u)}{t_{s_n} u} u^2 \, dx \to 0,
\end{align*}
a contradiction. Hence $t_0 > 0$. Again, by Corollary \ref{cor:2.3},
\begin{align*}
t_{s_n}^2 \| u \|_{s_n}^2 \to t_0^2 \|u \|^2 \quad \mathrm{as} \ n\to +\infty.
\end{align*}
Moreover, in view of Remark \ref{rem:1.1},
\begin{align*}
 \left| f(x,t_{s_n} u) t_{s_n} u \right| \leq  \varepsilon t_{s_n}^2 |u|^2 + C_\varepsilon t_{s_n}^p |u|^p \leq C ( |u|^2 + |u|^p)
\end{align*}
for some constant $C > 0$, independent of $n$. In view of the Lebesgue's convergence theorem,
\begin{align*}
\int_{\R^N} f(x,t_{s_n} u) t_{s_n} u \, dx \to \int_{\R^N} f(x,t_0 u) t_0 u \, dx.
\end{align*}
Thus the limit $t_0$ satisfies
\begin{align*}
t_0^2 \|u\|^2 = \int_{\R^N} f(x,t_0 u) t_0 u \, dx.
\end{align*}
Taking the Nehari identity into account we see that $t_0 = 1$. Hence
$t_s \to 1$ as $s \to 1^-$. Repeating the same argument we see that
\begin{multline*}
\limsup_{s \to 1^-} \left\{ \frac12 \int_{\R^N} f(x,m_s(u))m_s(u) - 2
  F(x,m_s(u)) \, dx \right\} \\
  = \frac12 \int_{\R^N} f(x,u)u - 2 F(x,u)
                               \, dx = \cJ(u) = c
\end{multline*}
and the proof is completed.
\end{proof}
\begin{Lem}\label{lemma-boundedness}
There exists a constant~$M > 0$ such that
\begin{align*}
\|u_s\|_{L^2 (\R^N)} + \|u_s\|_{s} + \|u_s\|_{L^\frac{2N}{N-1} (\R^N)} \leq M
\end{align*}
for every $s \in (1/2,1)$.
\end{Lem}
\begin{proof}
  Note that $\|u_s\|_{L^2 (\R^N)} + \|u_s\|_{L^\frac{2N}{N-1} (\R^N)} \leq C \|u_s\|_{s}$, for some $C > 0$ independent of $s$. So it is enough to
  show that $\|u_s\|_{s} \leq M$. Suppose by contradiction that
\begin{align*}
\|u_s\|_{s} \to +\infty \quad \mathrm{as} \ s \to 1^-.
\end{align*}
Put $v_s := \frac{u_s}{\|u_s\|_s}$. Then $\|v_s\|_{s} = 1$. In particular, $\{v_s\}$ is bounded in $L^2 (\R^N)$. Suppose that
\begin{equation}\label{eq:lions}
\sup_{y \in \R^N} \int_{B(y,1)} |v_s|^2 \, dx \to 0
\end{equation}
Then $v_s \to 0$ in $L^p (\R^N)$. Fix any $t > 0$. By \eqref{eq:5.1} we obtain
\begin{align*}
\cJ_{s_n} (u_{s_n}) \geq \cJ_{s_n} \left( \frac{t}{\|u_{s_n}\|_{s_n}} u_{s_n} \right) = \cJ_{s_n} (t v_{s_n}) = \frac{t^2}{2} - \int_{\R^N} F(x,tv_{s_n}) \, dx.
\end{align*}
From Remark \ref{rem:1.1} we see that
\begin{multline*}
\int_{\R^N} F(x,tv_{s_n}) \, dx \leq \varepsilon t^2 \| v_{s_n}\|_{L^2(\R^N)}^2 + C_\varepsilon t^p \|v_{s_n}\|_{L^p (\R^N)}^p \\
\to \varepsilon t^2 \limsup_{n \to \infty} \| v_{s_n}\|_{L^2(\R^N)}^2
\end{multline*}
for every $\varepsilon > 0$. Thus $\int_{\R^N} F(x,tv_{s_n}) \, dx  \to 0$ and for any $t > 0$
\begin{align*}
\cJ_{s_n} (u_{s_n}) \geq \frac{t^2}{2} + o(1),
\end{align*}
which is a contradiction with the boundedness of $\{\cJ_{s_n}(u_{s_n})\}_n$. Hence \eqref{eq:lions} does not hold, i.e. there is a sequence $\{z_n\} \subset \Z^N$ such that
\begin{align*}
\liminf_{n\to\infty} \int_{B(z_n,1+\sqrt{N})} |v_n|^2 \, dx > 0.
\end{align*}
or, equivalently
\begin{align*}
\liminf_{n\to\infty} \int_{B(0,1+\sqrt{N})} |v_n(x-z_n)|^2 \, dx > 0.
\end{align*}
From Theorem \ref{local-compactness}, $v_n (\cdot - z_n) \to v_0$ in $L^2_{\mathrm{loc}} (\R^N)$ and pointwise a.e., moreover $v_0 \neq 0$. See that, for a.e. $x \in \supp v_0$ we have
\begin{align*}
|u_{s_n} (x-z_n)| = \| u_{s_n} \|_{s_n} | v_{s_n} (x - z_n) | \to +\infty.
\end{align*}
Thus
\begin{align*}
o(1) = \frac{\cJ_{s_n} (u_{s_n})}{\|u_{s_n}\|_{s_n}^2} &= \frac12 - \int_{\R^N} \frac{F(x, u_{s_n})}{u_{s_n}^2} v_{s_n}^2 \, dx \\ &= \frac12 - \int_{\R^N} \frac{F(x, u_{s_n}(x-z_n))}{u_{s_n}(x-z_n)^2} v_{s_n}(x-z_n)^2 \, dx \\
&\leq \frac12 - \int_{\supp v_0} \frac{F(x, u_{s_n}(x-z_n))}{u_{s_n}(x-z_n)^2} v_{s_n}(x-z_n)^2 \, dx \to -\infty,
\end{align*}
a contradiction.
\end{proof}

\begin{Lem}\label{neh-bounded}
Since $u_s \in \cN_s$ there is (independent of $s$) constant $\rho$ such that
\begin{align*}
\|u_s\|_{s} \geq \rho > 0.
\end{align*}
\end{Lem}
\begin{proof}
  Since $u_s \in \mathcal{N}_s$, we can write by Remark \ref{rem:1.1}
  \begin{align*}
    \|u_s\|_{s}^2 = \int_{\R^N} f(x,u_s)u_s \, dx &\leq \varepsilon \|u_s\|_{L^2(\R^N)}^2 + C_\varepsilon \|u_s\|_{L^p(\R^N)}^p \\
                                                    &\leq C \left( \varepsilon \|u_s\|_{s}^2 + C_\varepsilon \|u_s\|_{s}^p \right)
  \end{align*}
  for a constant $C>0$ independent of $s$. Choosing $\varepsilon>0$
  small enough, we conclude that
  \begin{align*}
    \|u_s\|_{s}^{p-2} \geq \frac{1-C \varepsilon}{C \cdot C_\varepsilon} = \rho>0.
	\end{align*}
\end{proof}

\begin{Cor} \label{cor:2.9}
There exist $u_0 \in H^1 (\R^N)$, a sequence $\{z_n\}_n \subset \Z^N$ and a sequence $\{s_n\}_n$ such that $s_n \to 1^-$ and
\begin{align*}
u_{s_n} (\cdot - z_n) \to u_0 \neq 0 \quad \mathrm{in} \ L^\nu_{\mathrm{loc}} (\R^N) \quad \mathrm{as} \ n \to +\infty
\end{align*}
for all~$\nu \in [2, 2N/(N-1))$.
\end{Cor}
\begin{proof}
From Lemma \ref{lemma-boundedness} and Theorem \ref{local-compactness} we note that
\begin{align*}
u_{s_n} \to u_0 \quad \mathrm{in} \ L^\nu_{\mathrm{loc}} (\R^N) \quad \mathrm{as} \ n \to +\infty
\end{align*}
for all~$\nu \in [2, 2N/(N-1))$. If $u_0 \neq 0$, we can take $z_n = 0$ and the proof is completed. Otherwise $u_{s_n} \to 0$ in $L^2_{\mathrm{loc}} (\R^N)$ and therefore, $u_{s_n} (x) \to 0$ for a.e. $x \in \R^N$. Assume that
\begin{align*}
\sup_{y \in \R^N} \int_{B(y,1)} |u_{s_n}|^2 \, dx \to 0.
\end{align*}
Then from Theorem \ref{unif-lions} we know that $u_{s_n} \to 0$ in $L^\nu (\R^N)$ for all~$\nu \in [2, 2N/(N-1))$. Then
\begin{align*}
\int_{\R^N} f(x, u_{s_n}) u_{s_n} \, dx \to 0
\end{align*}
and $\|u_{s_n}\|_{s_n}^2 = \int_{\R^N} f(x, u_{s_n}) u_{s_n} \, dx \to 0$, which is a contradiction with Lemma \ref{neh-bounded}. Hence there is a sequence $\{ z_n \} \subset \Z^N$ such that
\begin{equation}\label{liminf}
\liminf_{n \to +\infty} \int_{B(0,1+\sqrt{N})} | u_{s_n} (\cdot - z_n)|^2 \, dx > 0.
\end{equation}
Moreover $\|u_{s_n} (\cdot - z_n)\|_{s_n} = \| u_{s_n} \|_{s_n}$, so that $\|u_{s_n} (\cdot - z_n)\|_{s_n}$ is bounded (see Lemma \ref{lemma-boundedness}). Hence, in view of Theorem \ref{local-compactness}
\begin{align*}
u_{s_n (\cdot - z_n)} \to \tilde{u}_0 \quad \mathrm{in} \ L^\nu_{\mathrm{loc}} (\R^N) \quad \mathrm{as} \ n \to +\infty
\end{align*}
for some $\tilde{u}_0$. Moreover, in view of \eqref{liminf}, $\tilde{u}_0 \neq 0$.
\end{proof}
\begin{Lem}\label{is-solution}
The limit $u_0 \in H^1 (\R^N) \setminus \{0\}$ is a  weak solution for \eqref{eq:1.2}.
\end{Lem}
\begin{proof}
Take any test function $\varphi \in C_0^\infty (\R^N)$ and note that by \cite[Section 6]{W}
we have
\begin{align*}
\int_{\R^N} (-\Delta)^{s/2} u_{s_n} (-\Delta)^{s_n /2} \varphi \, dx = \int_{\R^N} u_{s_n} (-\Delta)^{s_n} \varphi \, dx.
\end{align*}
Moreover
\begin{multline*}
\left| \int_{\R^N} u_{s_n} (-\Delta)^{s_n} \varphi \, dx - \int_{\R^N} u_0 (-\Delta \varphi) \, dx \right|
 \\ = \left| \int_{\R^N} u_{s_n} \left( (-\Delta)^{s_n} \varphi - (-\Delta \varphi) \right) \, dx + \int_{\supp \varphi} (u_{s_n} - u_0) (-\Delta \varphi) \, dx \right| \\
\leq \left\|u_{s_n} \right\|_{L^2(\R^N)} \left\| (-\Delta)^{s_n} \varphi - (-\Delta \varphi) \right\|_{L^2(\R^N)} \\
{}+ \|  (-\Delta \varphi) \|_{L^2(\R^N)} \|u_{s_n} - u_0\|_{L^2(\supp \varphi)} \to 0.
\end{multline*}
Hence
\begin{align*}
\lim_{n \to +\infty} \int_{\R^N} (-\Delta)^{s_n /2} u_{s_n} (-\Delta)^{s_n /2} \varphi \, dx &= \int_{\R^N} u_0 (-\Delta \varphi) \, dx \\
&= \int_{\R^N} \nabla u_0 \cdot \nabla \varphi \, dx.
\end{align*}
Obviously
\begin{align*}
\lim_{n \to +\infty} \int_{\R^N} V(x) u_{s_n} \varphi \, dx = \lim_{n \to +\infty} \int_{\supp \varphi} V(x) u_{s_n} \varphi \, dx= \int_{\R^N} V(x) u_0 \varphi \, dx.
\end{align*}
Take any measurable set $E \subset \supp \varphi$ and note that, taking into account Remark \ref{rem:1.1},
\begin{multline*}
\int_E | f(x,u_{s_n}) \varphi | \, dx \leq \varepsilon \| u_{s_n} \|_{L^2 (\R^N)} \| \varphi \chi_E \|_{L^2 (\supp \varphi)} \\
{}+ C_\varepsilon \| u_{s_n} \|_{L^p (\R^N)}^{p-1} \| \varphi \chi_E \|_{L^p (\supp \varphi)}.
\end{multline*}
Hence the family $\{ f(\cdot, u_{s_n}) \varphi \}_n$ is uniformly integrable on $\supp \varphi$ and in view of the Vitali convergence theorem
\begin{align*}
\lim_{n \to +\infty} \int_{\R^N} f(x,u_{s_n})\varphi \, dx = \int_{\R^N} f(x,u_0)\varphi \, dx.
\end{align*}
Therefore $u_0$ satisfies
\begin{align*}
\int_{\R^N} \nabla u_0 \cdot \nabla \varphi \, dx + \int_{\R^N} V(x) u_0 \varphi \, dx = \int_{\R^N} f(x,u_0)\varphi \, dx,
\end{align*}
i.e. $u_0$ is a weak solution to \eqref{eq:1.2}.
\end{proof}

\begin{proof}[Proof of Theorem \ref{th:main}]
Recalling Corollary \ref{cor:2.9} and Lemma \ref{is-solution}, it is sufficient to check that $u_0$ is a ground state solution, i.e. $\cJ(u_0) = c$. From Lemma \ref{is-solution} it follows that $u_0 \in H^1 (\R^N) \setminus \{0\}$ is a weak solution, so that $u_0 \in \cN$. Note that, from Corollary \ref{cor:2.9} and Fatou's lemma,
\begin{align*}
\liminf_{n \to +\infty} c_{s_n} &= \liminf_{n \to +\infty} \cJ_{s_n} (u_{s_n}) = \liminf_{n \to +\infty} \left\{ \cJ_{s_n} (u_{s_n}) - \frac12 \cJ_{s_n} ' (u_{s_n})(u_{s_n}) \right\} \\
&= \liminf_{n \to +\infty} \left\{ \frac{1}{2} \int_{\R^N} f(x,u_{s_n})u_{s_n} -
  2 F(x,u_{s_n}) \, dx  \right\} \\
  &=  \liminf_{n \to +\infty} \left\{ \frac{1}{2} \int_{\R^N} f(x,u_{s_n}(\cdot - z_n))u_{s_n}(\cdot - z_n) -
  2 F(x,u_{s_n}(\cdot - z_n)) \, dx  \right\} \\
&\geq  \frac12 \int_{\R^N} f(x,u_0)u_0 - 2 F(x,u_0) \, dx = \cJ(u_0) \geq c.
\end{align*}
Taking into account Lemma \ref{lem:limsup} we see that
\begin{align*}
c \leq \cJ(u_0) \leq \liminf_{n \to +\infty} c_{s_n} \leq \limsup_{n \to +\infty} c_{s_n} \leq c
\end{align*}
Hence $\lim_{n \to +\infty} c_{s_n}$ exists and $\lim_{n \to +\infty} c_{s_n} = c = \cJ(u_0)$.
\end{proof}

\section*{Acknowledgements}

The authors would like to thank an anonymous referee for several valuable comments helping to improve the original version of the manuscript. Bartosz Bieganowski was partially supported by the National Science Centre, Poland (Grant No. 2017/25/N/ST1/00531). Simone Secchi is member of the \emph{Gruppo Nazionale per l'Analisi Ma\-te\-ma\-ti\-ca, la Probabilit\`a e le loro Applicazioni} (GNAMPA) of the {\em Istituto Nazionale di Alta Matematica} (INdAM).


\begin{thebibliography}{99}

\bibitem{B} O. G. Bakunin, \textit{Turbulence and Diffusion: Scaling Versus Equations}, Springer, Berlin, 2008.

\bibitem{BHS} U. Biccari, V. Hern\'andez-Santamar\'{\i}a, \emph{The Poisson equation from non-local to local}, Electronic Journal of Differential Equations Vol. \textbf{2018} (2018), no. 145, 1--13.

\bibitem{Bi} B. Bieganowski, \textit{Solutions of the fractional Schr\"{o}dinger equation with a sign-changing nonlinearity}, J. Math. Anal. Appl. \textbf{450}, Issue 1 (2017), 461--479.

\bibitem{BM} B. Bieganowski, J. Mederski, \textit{Nonlinear Schr\"{o}dinger equations with sum of periodic and vanishing potentials and sign-changing nonlinearities}, Commun. Pure Appl. Anal., Vol. \textbf{17}, Issue 1 (2018), 143--161.

\bibitem{BS} B. Bieganowski, S. Secchi, \textit{Non-local to local transition for ground states of fractional Schr\"odinger equations on bounded domains}, to appear in Topol. Meth. Nonlin. Anal.

\bibitem{BC} J.P. Borthagaray, P. Ciarlet Jr., \emph{On the convergence in the \(H^1\)-norm for the fractional laplacian}, arXiv:1810.07645.

\bibitem{BBM} J. Bourgain, H. Brezis, P. Mironescu, \textit{Another
    look at Sobolev spaces}, Optimal Control and Partial Differential
  Equations, 439--455, IOS, Amsterdam, 2001.
  
\bibitem{Brezis} H. Brezis. Functional analysis, Sobolev spaces and partial differential equations. Springer, 2011.

\bibitem{KT} A. Cotsiolis, N. Tavoularis, \emph{Best constants for
      Sobolev inequalities for higher order fractional derivatives},
    J. Math. Anal. Appl. \textbf{295} (2004), 225--236.

\bibitem{DNPV} E. Di Nezza, G. Palatucci, E. Valdinoci,
    \emph{Hitchhiker's guide to the fractional {S}obolev spaces},
    Bull. Sci. Math. \textbf{136} (2012), 521--573.

\bibitem{DPV}    S. Dipierro, G. Palatucci, E. Valdinoci, \textit{Dislocation dynamics in crystals: A macroscopic theory in a fractional Laplace
setting}, Comm. Math. Phys. \textbf{333} (2015), no. 2, 1061--1105.

\bibitem{dPKSz} F.O. de Paiva, W. Kryszewski, A. Szulkin, \textit{Generalized Nehari manifold and semilinear Schr\"odinger equation with weak monotonicity condition on the nonlinear term}, Proc. Am. Math. Soc. \textbf{145} (2017), no. 11, 4783--4794.

\bibitem{G} G. Gilboa, S. Osher, \textit{Nonlocal operators with applications to image processing}, Multiscale Model. Simul. \textbf{7} (2008),
no. 3, 1005--1028.

\bibitem{MBRS} G. Molica Bisci, V. Radulescu, R. Servadei, \textit{Variational methods for nonlocal fractional problems}. Encyclopedia of Mathematics and Its Applications 162, Cambridge University Press, 2016.

\bibitem{S1} S. Secchi, \textit{Ground state solutions for nonlinear fractional Schr\"odinger equations in $\R^N$}, Journal of Mathematical Physics \textbf{54}, 031501 (2013).

\bibitem{S2} S. Secchi, \textit{On fractional Schrödinger equations in $\R^N$ without the Ambrosetti-Rabinowitz condition}, Topol. Methods Nonlinear Anal. \textbf{47} (2016), 19--41.

\bibitem{SzW} A. Szulkin, T. Weth, \textit{Ground state solutions for some indefinite variational problems}, J. Funct. Anal. \textbf{257} (2009), no. 12, 3802--3822.

\bibitem{SzW-Handbook} A. Szulkin, T. Weth, \textit{The method of Nehari manifold}, in: David Yang Gao, Dumitru Motreanu [ed.], Handbook of Nonconvex Analysis and Applications, Boston: International Press, 2010, 597--632.

\bibitem{V} J.L. V\'azquez, \textit{Nonlinear diffusion with fractional Laplacian operators}, in: Nonlinear Partial Differential Equations (Oslo
2010), Abel Symp. 7, Springer, Heidelberg (2012), 271--298.

\bibitem{W} M. Warma, \textit{The fractional relative capacity and the
  fractional Laplacian with Neumann and Robin boundary conditions on
  open sets}, Potential Anal. \textbf{42} (2015), no. 2, 499--547.


\end{thebibliography}
\end{document}